\documentclass[12pt,oneside]{article}
\usepackage[cp1251]{inputenc}
\usepackage[T2B]{fontenc}
\usepackage[russian,english]{babel}

\usepackage{amsmath}
\usepackage{amssymb}
\usepackage{amsthm}

\usepackage{floatflt}

\usepackage{makeidx}
\usepackage{stmaryrd} 
\usepackage{mathrsfs} 

\usepackage{cmap}

\ifnum\pdfoutput<1
\usepackage[dvips]{graphicx}
\DeclareGraphicsExtensions{.eps,.ps,.eps.gz,.ps.gz,.eps.Z,.bmp}
\fi
\ifnum\pdfoutput=1
\usepackage[pdftex]{graphicx}
\fi

\usepackage[hypertex, unicode, linktocpage, hyperindex, hyperfootnotes=false, pagebackref, pdfstartview={XYZ null null 0.9}]{hyperref}

\graphicspath{{pic/}}
\graphicspath{{../pic/}}

\newcounter{point}[subsection]
\newcounter{ppoint}[section]
\newcommand{\p}{\vspace{0.0cm}\par\noindent \addtocounter{point}{1}\thepoint. }
\newcommand{\pp}{\vspace{0.0cm}\par\noindent \addtocounter{ppoint}{1}\theppoint. }

\newtheoremstyle{NewTheoremStyle}
  {9pt}
  {9pt}
  {\itshape}
  { }
  {\bfseries}
  {.}
  { }
  { }

\theoremstyle{NewTheoremStyle}
\newtheorem{Theorem}{Theorem}

\newcounter{Example}
\newenvironment{Example}{\stepcounter{Example}\par \noindent \textbf{Example \theExample.} \noindent}{\hfill$\qed$\newline }

\renewcommand{\G}{\mathcal{G}}
\renewcommand{\H}{\mathcal{H}}

\renewcommand{\L}{\mathcal{L}}

\newcommand{\ad}{\mathrm{ad}\,}
\newcommand{\Ad}{\mathrm{Ad}\,}
\newcommand{\Ker}{\mathrm{Ker}\,}
\newcommand{\ind}{\mathrm{ind}\,}

\newcommand{\rank}{\mathrm{rank}\,}

\renewcommand{\G}{\mathfrak{g}}
\renewcommand{\H}{\mathfrak{h}}

\renewcommand{\L}{\mathfrak{l}}
\newcommand{\K}{\mathfrak{k}}
\newcommand{\N}{\mathfrak{n}}

\newcommand{\pder}[2]{\frac{\partial\, #1}{\partial\, #2}}

\newcommand{\Le}{\left}
\newcommand{\Ri}{\right}
\newcommand{\la}{\langle}
\newcommand{\ra}{\rangle}

\renewcommand{\mathbf}{\boldsymbol}
\newcommand{\corank}{\mathrm{corank}\,}

\title{\textbf{Algebraic method for construction\\of infinitesimal invariants of Lie groups\\representations}}
\author{Oleg L. Kurnyavko\thanks{Omsk Institute of Water Transport, Omsk, Russia, kurnyavko@mail.ru}, 
Igor V. Shirokov\thanks{Omsk State Technical University, Omsk, Russia, iv\_shirokov@mail.ru}}

\begin{document}
\maketitle

\begin{abstract}
\noindent \emph{We propose the method for obtaining invariants of arbitrary representations of Lie groups  that reduces this problem to known problems of linear algebra. The basis of this method is the idea of a special extension of the representation space, which allows us to regard it as a coalgebra of some Lie algebra. In its turn, this allows us to reduce the problem of constructing invariants of a given representation to the problem of constructing invariants of the coadjoint representation of the corresponding Lie group. In our previous paper it was shown that this problem can always be solved in a natural way by the methods of linear algebra.}
\end{abstract}

\newpage
\tableofcontents

\section{Introduction}
\pp Let $V$ be an arbitrary set and $G$ be some group, for which an action on the set $V$ is defined, i.e.
\begin{gather*}
    G\times V\to V\colon\ x\to y=g\, x,\quad x,y\in V,\ g\in G,
\end{gather*}
then the function $J(x)$ on the set $V$, such that $J(g\,x) = J(x)$, is called the \emph{ invariant of this group action} \cite{Vinberg:1989}. One would expect that the main task of the \emph{invariant theory} should be constructing a general method for obtaining a complete system of invariants for an arbitrary set $V$, an arbitrary group $G$, and an arbitrary action of a given group on the indicated set. However, it is obvious that in such a general formulation the problem can not be solved because of possible significant differences both in the structure of the set $V$ and the group $G$ and in the indicated set of the action of $G$ on $V$.

Even for the given set $V$, group $G$ and a certain action $G$ on $V$, the construction of a complete system of invariants is a very difficult task. In each concrete case, the main problem of invariant theory involves solving the following problems: the proof of the finite generation of the required invariants set and the development of a practical method for their obtaining. In many cases it is possible to solve the first problem successfully , and for the latter, researchers often have to confine themselves to proving general non-constructive existence theorems, and so on.

\pp\emph{ Classical invariant theory } (more precisely, the algebraic theory of invariants) is restricted to the problem of constructing polynomial invariants of linear representations, i.e. the set $V$ is a finite-dimensional linear space, on which the group $G$ acts by some of its linear representations. The main results of this theory are primarily connected with two fundamental works of Hilbert: \cite{Hilbert:1890}, \cite{Hilbert:1893}. In the first work, the finiteness of the number of the integer invariants generating rings was proved, but it possessed a significant disadvantage by the standards of that time --- it was not constructive. In this connection, a second paper appeared, in which estimates were obtained from above on the degree of basis invariants. A complete solution of the above problem is possible only in a number of special cases. So in the second half of the 19th century, an exhaustive answer was found for invariants of binary forms of degree less than 6 and only comparatively recently for invariants of binary forms of degree 8 \cite{Shioda:1967}. In addition to these cases, there is essentially only one systematically studied class of linear groups for which the problem is satisfactorily solved: the class of groups with a free algebra of invariants \cite{Shephard:1954}, \cite{Kac:1976}, \cite{Schwarz:1978}, \cite{Adamovich:1979}. The most general result was obtained in the papers \cite{Popov:1981}, \cite{Popov:1982}, which gives a principal way of explicitly describing these G-modules and their invariant algebras.

\pp \emph{Modern invariant theory } is based largely on the use of geometric methods, which allowed us to abandon the need to restrict the study to polynomial invariants. The price of such an extension was that, in most cases, the results of this theory are non-constructive in nature, i.e., are existence theorems, and so on. However, numerous practical applications, and primarily related to the current problems of quantum field theory, \emph{lead to the need to develop methods for explicitly constructing invariants}.

We note the important concept of the \emph{differential invariant}, which arose in the framework of group analysis of differential equations. Some important results related to differential invariants were obtained in the papers \cite{Olver:2001}, \cite{Olver:2007}, \cite{Olver:2010}, \cite{Shirokov:2007eng}, \cite{Shirokov:2015eng}.

One of the important cases admitting a constructive solution of the basic problem of invariant theory is related to the constructing invariants of representations of Lie groups. Let $G$ be a connected Lie group, $T(G)$ be some representation of the group $G$ in the linear space $V$, then the invariants of the given representation are defined by the following conditions:
\begin{gather*}
    J(T_g\,x)=J(x),\quad x\in V,\ g\in G,\ J(x)\in C^\infty(V),
\end{gather*}
where $T_g$ are the operators realizing the representation of the group $G$ on the linear space $V$. Their construction reduces to solving a system of differential equations:
\begin{gather}\label{eq:081}
    -t_{Ab}^a\,x^b\,\pder{J(x)}{x^a}=0,
\end{gather}
where $t_{Ab}^a = \pder{(T_g)^a_b}{g^A} \Big |_{g = e}, \ A=\overline {1, \dim G}, \ a, b = \overline {1, \dim V} $, and $ t_{Ab}^a$ are elements of the matrices forming the Lie algebra representation basis of the given Lie group. Despite the fact that there is no general method for integrating systems of differential equations, in most cases, there are techniques that, in principle, make it possible to find the desired solution. However, technical difficulties along this path can be very significant, which forces researchers either to look for cases for possible integration of the system (see, for example, \cite{Berzin:1996}), or to develop methods for solving it, unrelated to their direct integration (but their application, as a rule, is also limited to special types of Lie groups).

The case of constructing invariants of the coadjoint representation is the most studied and significant for applications in mathematics and physics. In particular, finding the Casimir functions is an important stage in the study of completely integrable Hamiltonian systems, which can be considered as classical systems on the orbits of the coadjoint representation \cite{Mish'enko:1976}, \cite{Bolsinov:1991}.

The problem of constructing invariants of the coadjoint representation for semisimple Lie algebras was largely solved in the papers \cite{Casimir:1931}, \cite{Racah:1950} and \cite{Chevalley:1951}. Invariants of classical Lie algebras were obtained in the papers \cite{Gelfand:1950}, \cite{Perelomov:1966}, \cite{Perelomov:1967}. For non-semisimple Lie algebras up to the present, in the general case the problem has not been solved. The main results refer to numerous special cases, see, for example, \cite{CampoamorStursberg:2002}. An important result was obtained for the case of solvable Lie algebras of arbitrary dimension, namely, in the paper \cite{Boyko:2006}, an algebraic algorithm for constructing invariants of the coadjoint representation was obtained on the basis of the moving frame method.

A general method for constructing invariants of coadjoint representation of Lie groups was proposed by us in the paper \cite{Shirokov:2016eng}. This method naturally reduces the problem of constructing invariants to the known problems of linear algebra. This method can be used equally well for Lie groups with arbitrary structure.

\pp The present work is devoted to the development of a general method for constructing infinitesimal invariants of an arbitrary  representation of Lie group, which also reduces this problem to the known problems of linear algebra, eliminating the need for the direct integration of the system (\ref{eq:081}).


In the second section we briefly describe the method of constructing invariants of the coadjoint representation of Lie groups, which was first introduced by us in the paper \cite{Shirokov:2016eng} and which is key to understanding the general method for constructing infinitesimal invariants of an arbitrary representation of a Lie group.

In the third section we propose a generalization of the method of constructing invariants of a coadjoint representation that allows one to construct infinitesimal invariants of an arbitrary representation of a Lie group, including, naturally, the invariants of the adjoint and coadjoint representations. Also here we will consider the question of the relationship between invariants of representations of Lie groups and their conjugate (dual) representations, which, in particular, in a number of cases makes it possible to construct a complete system of invariants of the coadjoint representation in a simple way, having a complete system of invariants of the adjoint representation, and conversely, to construct invariants of the adjoint representation by the invariants of the coadjoint representations.

In the fourth section we give example of the use of this method, i.e. the construction of invariants of a representation and invariants of the representation conjugate to a given one.

\section{Invariants of the coadjoint representation}
\subsection{The coadjoint representation of a Lie group}
Let $ G $ be a Lie group, $\G $ its Lie algebra, and $\G^* $ its coalgebra, i.e. for $\G$ and $\G^*$, pairing is defined as
\begin{gather*}
    \G\times \G^*\to \mathbb{R},\quad X, x\to \alpha=\la X, x \ra,\quad x\in\G,\ X\in\G^*,\ \alpha\in \mathbb{R}.
\end{gather*}
The mapping
\begin{gather}\label{eq:074}
    G\times\G^*\to\G^*\colon\quad (g,X)\to Y\equiv\mathrm{Ad}^*_g X,\quad g\in G,\ X,Y\in\G^*,\notag
    \intertext{where the action $\mathrm{Ad}^*_g$ is determined by the rule:}
    \la \mathrm{Ad}^*_g X, x\ra =\la X, \mathrm{Ad}_{g^{-1}} x\ra
\end{gather}
is called the \emph{coadjoint representation of the Lie group $G$}. The coadjoint representation operator has the form:
\begin{gather}\label{eq:078}
    (\mathrm{Ad}^*_g)_j^i X_i=(\mathrm{Ad}_{g^{-1}})_j^i X_i, \quad i,j=\overline{1,\dim \G}. 
\end{gather}
\emph{Invariants of a coadjoint representation} (\emph{also called Casimir functions}) are the functions on a Lie algebra that have the property:
\begin{gather*}
    J^*(\Ad^*_g\, X)=J^*(X).
\end{gather*}
It is not difficult to show that, by (\ref{eq:004}), the infinitesimal invariance criterion has the form:
\begin{gather}\label{eq:086}
    C_{ij}(X)\,\pder{J^*(X)}{X_j}=0, \quad i,j,k=\overline{1,\dim \G}.
\end{gather}
where $C_{ij}(X)=C_{ij}^k\,X_k$. The number of functionally independent invariants is determined by the rank of the matrix $C_{ij}(X)$ and is called the \emph{index of the Lie algebra} $\G$:
\begin{gather*}
    \ind \G=\dim \G^*-\sup\limits_{X\in\G^*}\rank C_{ij}(X).
\end{gather*}

The problem of constructing invariants of coadjoint representation using methods of linear algebra was described in detail in our paper \cite{Shirokov:2016eng}. However, given the importance of these results for understanding the general method of constructing invariants, which is proposed in this paper, we allow ourselves briefly to give them below. First of all, we recall the general algorithm for finding the polarization of a Lie algebra, which is one of the key components of this method, and then we will give the method for constructing invariants of the coadjoint representation.

\subsection{Construction of the Lie algebra polarization}
\p The construction of polarization is a well-known algebraic problem, however, there are some problems that are usually not considered by researchers. In this regard, we wil recall some important definitions and theorems related to this concept, and will also formulate a number of useful statements that will help the reader to solve this problem effectively .

Let $\G $ be the Lie algebra, $\G^*$ be the corresponding coalgebra, and $B_\lambda(X, Y)$ be the skew-symmetric form on $\G $ defined by
\begin{gather}\label{eq:103a}
    B_\lambda(X,Y)=\la \lambda, [X,Y] \ra,\quad X,Y\in \G,\quad \lambda\in\G^*,
\end{gather}
which has the following coordinate form:
\begin{gather*}
    B_\lambda(X,Y)=B_{ij}\,X^i\,Y^j,\quad B_{ij}=C_{ij}^k\,\lambda_k\equiv C_{ij}(\lambda),\quad i,j,k=\overline{1,\dim \G}.
\end{gather*}
The kernel of the form $ B_\lambda(X, Y) $, defined by
\begin{gather}\label{eq:104a}
    \Ker\, B_\lambda=\{X\in \G\ |\ B_\lambda(X,\G)=0\}\equiv \G^\lambda,
\end{gather}
coincides with the annihilator $\G^\lambda$ of the covector $\lambda\in\G^*$.

A subalgebra $\H $ such that
\begin{gather}\label{eq:022a}
    B_\lambda(\H, \H)=0,
\end{gather}
is called \emph{subordinate to the covector $\lambda\in~\G^*$}. A subalgebra $\H$ that is subordinate to the covector $\lambda\in\G^*$, which has the maximum possible dimension, i.e. $\dim\H =\frac{1}{2} \Le(\dim\G+\dim\G^\lambda\Ri)$, is called the \emph{polarization of the algebra $\G$ relative to the covector $\lambda$} or simply \emph{polarization of the covector (functional) $\lambda$}. Further we will use the notation $\H\equiv\mathcal{P}(\G)$. The problem of constructing polarization reduces to a considerable extent to the use of the corresponding theorems (see, for example, \cite{Diksmye:1974fr}) for the case of Lie algebras that are either semisimple or solvable, or unsolvable for which the polarization of the orthogonal complement to its radical is known.

\p In the case of a semisimple Lie algebra $\G$, its polarization with respect to the covector $\lambda\in\G^*$ is defined (see, for example, \cite{Diksmye:1974fr})
\begin{gather}\label{eq:009a}
    \mathcal{P}(\G)=\L^{(+)}\supsetplus \K,
\end{gather}
where $\K $ is some Cartan subalgebra, $K$ is a nondegenerate element in the subalgebra $\K $ corresponding to the covector $\lambda$, and $\L^{(+)}$ is the linear span of eigenvectors of the operator $\ad_K$ corresponding to positive eigenvalues. In the \cite{Shirokov:2016eng}, it was proved that the nondegenerate covector $\lambda\in\G^*$ corresponds to the nondegenerate vector $K\in\G $, which means that the expression (\ref{eq:009a}) determines the polarization of the semisimple Lie algebra $\G $ with respect to the nondegenerate covector $\lambda\in\G^*$.

An obvious but nonetheless important case that reduces to finding the polarization of a semisimple Lie algebra is the construction of  polarization of the reductive Lie algebra $\G$, i.e. $\G = Z \oplus S$, where $Z$ is the center in $\G $, and $ S $ is a semisimple subalgebra. It is easy to see that the polarization of the Lie algebra $\G$ is determined by
\begin{gather}\label{eq:013a}
    \mathcal{P}(\G)=\mathcal{P}(Z)\oplus\mathcal{P}(S),
\end{gather}
it is obvious that $\mathcal{P}(Z) = Z$.

\p The polarization of the solvable algebra $\G $ relative to the covector $\lambda\in\G^*$ is determined by the expression (see, for example, \cite{Diksmye:1974fr}):
\begin{gather}\label{eq:014a}
    \mathcal{P}(\G)=\N_{1}+\N_{2}+\ldots+ \N_{n},\quad \N_{i}=\ker B_i,\quad B_i=B_\lambda\big|_{\G_i},
\end{gather}
where $\G_i$ are the elements of a decreasing chain of ideals in $\G $
\begin{gather}\label{eq:010a}
    \G=\G_{n}\supset\G_{n-1}\supset\ldots\supset\G_{1}\supset \{0\},\quad \dim \G_i=i,\quad i=\overline{1,n}
\end{gather}
where $ n =\dim\G$. Note that if $B_\lambda = 0$ on $\G$, then $ \mathcal{P}(\G) = \G$. Thus, the problem of constructing the polarization reduces to constructing a chain (\ref{eq:010a}), which can be obtained by condensing the upper central series of the Lie algebra $\G$
\begin{gather}\label{eq:011}
    \G^{(0)}\supset\G^{(1)}\supset\ldots\supset\G^{(m-1)}\supset\G^{(m)},\ \G^{(m)}=[\G^{(m-1)}, \G^{(m-1)}],\ \G^{(0)}=\G,
\end{gather}
where the dimensions of the neighboring elements in the general case can differ by more than one. The compaction procedure was discussed in detail in \cite{Shirokov:2016eng}.

\p Let us consider the problem of constructing the polarization for an arbitrary Lie algebra. Let $\G$ be an arbitrary Lie algebra, $R$ its radical and $R^\perp$ be the orthogonal complement of $R$ in $\G$ with respect to the form $ B_\lambda (X,Y) $, i.e.
\begin{gather*}
    R^\perp=\{X\in \G\ |\ B_\lambda(X,R)=0\},
\end{gather*}
we note that for $ R^\perp$ the solvable polarization $\mathcal{P}(R^\perp)$ is known. Then there exists a polarization $ \mathcal{P}(R)$ of the radical $ R $ such that
\begin{gather}\label{eq:008a}
    [ \mathcal{R}(R^\perp), \mathcal{P}(R)]\subset \mathcal{P}(R),
\end{gather}
and the polarization $\mathcal {P}(\G)$ of the Lie algebra $\G$ is determined by the sum \cite{Diksmye:1974fr}:
\begin{gather}\label{eq:012a}
    \mathcal{P}(\G)=\mathcal{R}(R)+\mathcal{P}(R^\perp).
\end{gather}
This formula makes it possible to construct a polarization for an arbitrary Lie algebra. Thus, finally, the algorithm for constructing the polarization with respect to the given covector takes the following form:
\begin{itemize}
  \item we select a radical in the original algebra and its orthogonal complement;
  \item we construct the polarization of the orthogonal complement $R^\perp$; if $R^\perp$ is a semisimple, reductive or solvable algebra, then the question of constructing the polarization is solved by using the formulas (\ref{eq:009a}), (\ref{eq:013a}) or (\ref{eq:014a} ), respectively; If $ R^\perp$ is unsolvable but not semisimple, then using the Levi-Maltcev decomposition, we can select the radical and orthogonal complements to it in $R^\perp$ and apply this procedure again; as a result, we find the polarization $\mathcal{P}(R^\perp)$ and thus satisfy the first of the initial conditions for applying the formula (\ref{eq:012a}).
  \item we construct the polarization of the radical satisfying (\ref{eq:008a}), using (\ref{eq:014a}), since the radical by definition is a solvable Lie algebra; we recall that the problem of constructing a chain of ideals of codimension one is considered in detail in \cite{Shirokov:2016eng}.
\end{itemize}

In conclusion, we note that if the Lie algebra $\G $ is completely solvable, then for it there exists a polarization with respect to any points $\lambda \in\G^*$. Conversely, if $\G$ is an arbitrary solvable or semisimple Lie algebra, then there can exist such $\lambda\in\G^*$ for which polarization does not exist \cite[pp.~73,~83]{Diksmye:1974fr}.

\subsection{Construction of the Lie group coadjoint representation invariants}
\p The main idea of the method of constructing coadjoint invariants proposed in our paper \cite{Shirokov:2016eng} is that the construction of a complete set of invariants of the coadjoint representation of some group $ G $ is equivalent to constructing the canonical transition to the Darboux coordinates on the orbits of the coalgebra $\G^*$ of maximal dimension dual to the Lie algebra $ \G $ of the Lie group $G$.

Indeed, suppose that there is a Lie group $G$, its Lie algebra $\G$ and its corresponding coalgebra $\G^*$, and let each element $f\in\G^*$ corresponds to the coordinates $(f_1, \ldots, f_n)$, where $n =\dim\G^*$. The action of the group $G$ on the coalgebra $\G^* $ splits the latter into orbits $O_\lambda$, where $ \lambda\in\G^*$ is some element of the given orbit. Each of these orbits is a symplectic manifold, and therefore, by Darboux's theorem, special coordinates $(q, p)$ can be introduced on them such that the corresponding symplectic form has the form:
\begin{gather}\label{eq:007a}
    \omega_\lambda=dp_a\wedge dq^a,\quad a=~{1,\ldots,\frac{1}{2}\,\dim O_\lambda}.
\end{gather}
The transition to canonical coordinates is determined by functions of the form:
\begin{gather}\label{eq:005a}
    f_i=f_i(q,p,\lambda),\quad i=1,\ldots, \dim \G^*,
\end{gather}
such that $f_i(0,0, \lambda)=\lambda_i $. The coordinates $(q, p)$ can be considered as local coordinates of the surface, which is the orbit of $O_\lambda$, and the equations themselves (\ref{eq:005a}) as a parametric equation of this surface. Eliminating the variables $q$ and $p$ in these transition functions, we obtain the equation of the given orbit in the form
\begin{gather*}
    \psi_\mu(f)=\omega_\mu,\quad \mu=1,\ldots, n-\dim O_\lambda,
\end{gather*}
where $\omega_\mu=\psi_\mu(\lambda)$. Thus, we obtain $n-\dim O_\lambda $ functions for which $\psi_\mu(Ad^*_g f) = \psi_\mu(f) $ is true, i.e. these functions are invariants of the coadjoint representation.

If the covector $\lambda$ is non-degenerate, that is the orbit $O_\lambda$ has the maximum possible dimension, then the transformations (\ref{eq:005a}) define $\ind\G$ functions that are invariant under the action of the coadjoint representation, i.e. a complete set of solutions of the system (\ref{eq:086}).

Each nondegenerate orbit is determined by the set of numbers $\omega=(\omega_1, \ \omega_2, \ldots, \omega_{\ind\G})$, which in turn is determined by the choice of the covector $\lambda$ such that $\rank C_{ij}(\lambda)$ has the maximum possible value. In this case, $\omega_\alpha = \psi_\alpha(\lambda_1)=\psi_\alpha(\lambda_2)$, if $\lambda_{1,2} \in O_{\lambda}$.  Therefore it is advisable to introduce the parametrization of $\lambda(j)$ orbits $O_\lambda $, where $j =(j_1, j_2, \ldots, j_{\ind\G})$ is a set of numbers such that different sets of parameters $j$ correspond to covectors $\lambda$ belonging to different orbits $O_\lambda$. Consequently, this parameterization can be determined from the condition:
\begin{gather}\label{eq:027}
    \psi_\mu(\lambda(j))=\omega_\mu(j),\qquad \det\pder{\omega_\mu(j)}{j}\neq 0.
\end{gather}
It is obvious that the left-hand side of the last inequality can be rewritten in the form:
\begin{gather*}
    \det\pder{\omega_\mu(j)}{j_\alpha}=\det\pder{\psi_\mu}{f_i}\,\pder{f_i(j)}{j_\alpha}\,\Big|_{f=\lambda(j)}=
    \det X_\mu^{i}\,\pder{f_i(j)}{j_\alpha},
\end{gather*}
where $X_\mu^i$ are the components of the vectors that form the basis of the annihilator of the covector $\lambda(j)$. Parametrization of $\lambda(j)$ can be chosen in a linear way, i.e. $\lambda_i(j)=~a_i^\alpha\, j_\alpha$, then the inequality (\ref{eq:027}) takes the form:
\begin{gather}\label{eq:028}
    \det X_\mu^{i}(\lambda(j))\,a_i^\alpha \neq 0.
\end{gather}
Taking into account that the construction of the annihilator is an elementary problem of linear algebra, the inequality (\ref{eq:028}) allows us to choose a set of parameters $a_i^\alpha$ that provide the required parametrization.

Thus, we obtain a simple way of calculating the invariants from the equations (\ref{eq:005a}). Indeed, using the parametrization described above $\lambda=\lambda(j)$, we obtain the relations:
\begin{gather*}
    f_i=f_i(q,p,j),\quad i=1,\ldots, \dim \G^*,
\end{gather*}
and by eliminating the variables $q$ and $p$ in that relations we obtain equations of the form:
\begin{gather*}
    \varphi_\mu(f, j)=0\quad \mu=1,\ldots, \ind\G,
\end{gather*}
from which we can obtain the parameters $j_\alpha$ and obtain the relations:
\begin{gather*}
    j_\alpha=j_\alpha(f),\quad \alpha=1,\ldots, \ind\G.
\end{gather*}
These relations are the desired invariants of the coadjoint representation of the Lie group $ G $.

\p Reduction of the problem of constructing invariants to the problem of constructing transition functions to the Darboux coordinates is justified by the fact that the latter is solved elementary by linear algebra methods for the case of a canonical transition, linear in variables $p_a$ on an arbitrary non-degenerate orbit $ O_\lambda$ of the coalgebra $\G^*$ , i.e. for the transition functions of the form:
\begin{gather}\label{eq:006a}
    f_i(q,p)=\alpha^a_i(q)\,p_a+\chi_i(q),\quad \rank \alpha^a_i(q) =\frac{1}{2}\,\dim O_\lambda.
\end{gather}
Of course, there is no linear transition (\ref{eq:006a}) in the general case, but if we assume that $\alpha_i^a(q) $ and $ \chi_i(q)$ are holomorphic functions of complex variables $ q $, then we greatly extend the class of Lie algebras and corresponding orbits of the coadjoint representation for which this transition exists (here it is necessary to assume that the linear functionals in $ \G^* $ extend to the complex algebra $\G^c $ by linearity). We also note that, apparently, for every Lie algebra and for any of its nondegenerate orbits the transition (\ref{eq:006a}) exists.

In the \cite{Shirokov:2000eng} it was shown that if for the Lie group $G$ the left-invariant vector fields are known
\begin{gather*}
    \xi_i=\xi_i^a(q)\,\partial_{q^a}+\xi_i^A(h,q)\,\partial_{h^A},
\end{gather*}
which are written with respect to a special basis in the tangent space corresponding to the basis of the Lie algebra $\G=\{l_A, l_a\} $ of the group $G$, where the elements of the basis $l_A$ correspond to the polarization basis $\mathcal{P}(\G)$ of the Lie algebra $\G$, then the desired linear transition to the Darboux coordinates on a nondegenerate orbit $O_\lambda$ is defined :
\begin{gather}\label{eq:026a}
    f_i(q,p)=\xi^a_i(q)\,{p_a}+\xi^A_i(0, q)\,\lambda_A.
\end{gather}
In fact, the transition functions $ f_i $ that we are interested in are obtained from left-invariant fields by replacing $\partial_{h^A} \to \lambda_A, \ \partial_{q^a} \to p_a, \ h = 0 $. A direct substitution shows that the transition functions (\ref{eq:026a}) give the symplectic form on the orbit under consideration to the form (\ref{eq:007a}).

\p In conclusion, we give the final algorithm for constructing invariants of the coadjoint representation of the Lie group. Suppose that there is a Lie group $G$, its Lie algebra $\G =\{e_i\}$ and the corresponding coalgebra $\G^*= \{e^i\}$:
\begin{itemize}
  \item we construct the polarization $\H=\mathcal{P}(\G)=\{l_A\}$ for some non-degenerate covector $\lambda \in\G^*$;
  \item we construct a basis of the Lie algebra $\G$, complementing the basis of the algebra $\H$ in an appropriate way, i.e. $\G =\{l_A, l_a\}$;
  \item we construct the left-invariant fields $\xi_i(h, q) = \xi_i^a(q)\,\partial_{q^a}+\xi_i^A(h, q)\, \partial_{h^A}$ in the basis $\{l_A, l_a\}$;
  \item we construct the Darboux coordinates on $\G$ using the rule
  $$F_i(q,p)=\xi_i(h,q)\Big|_{{h=0 \atop \partial_{q^a}\to p_a} \atop \partial_{h^A}\to \Lambda_A}=\xi^a_i(q)\,{p_a}+\xi^A_i(0, q)\,\Lambda_A;$$
  \item performing the transitions $F_i\to f_i$ and $\Lambda_i\to\lambda_i$, where $f_i$ and $\lambda_i$ are the coordinates of the covectors corresponding to the original basis of the Lie algebra $\G=\{e_i\}$ and then excluding the variables $q_a$ and $p_a$ in the expressions for the transition functions we obtain the expressions of the form $j_\alpha=j_\alpha(f)$, which are the desired invariants of the coadjoint representation of the Lie group $G$.
\end{itemize}
We note once again that the problem of constructing the left-invariant fields $ \xi_i$ for arbitrary Lie algebras in the given algorithm is also solved by the methods of linear algebra \cite{Shirokov:1997eng} and, moreover, has been realized by the authors in computer mathematics systems such as \verb"Waterloo Maple " and \verb" Wolfram Mathematica".


\section{General method for construction invariants}
\subsection{Representations of Lie groups}
\p Let $G$ be a Lie group, and let $T(G)$ be its representation in the space $V $, i.e.
\begin{gather*}
    G\ni g \to T_g\in T(G),\quad    V \ni x \to y=T_g\,x \in V,
\end{gather*}
where $(T_g\,x)^a=(T_g)^a_b\,x^b$ и $a,b=\overline{1,\dim V}$ and besides
\begin{gather*}
    T_e=E,\quad T_{g_1\, g_2}=T_{g_1}\,T_{g_2}.
\end{gather*}
i.e. $(T_e)^a_b=\delta^a_b$, $(T_{g_1\, g_2})^a_b=(T_{g_1})^a_c\,(T_{g_2})^c_b$. We recall that these properties imply that
\begin{gather*}
    T_{g^{-1}}=(T_g)^{-1}.
\end{gather*}

\p Every representation of the Lie group $G$ on the space $V$ induces in the space $C^{\infty}(V)$ \emph{left regular representation}:
\begin{gather*}
    \G\ni g \to T^L_g\in T^L(\G),\\
    C^{\infty}(V)\ni f(x) \to  F(x)=f(T_g\,x)\equiv T^L_g\, f(x)\in C^{\infty}(V).
\end{gather*}
Functions on the representation space $V$ that satisfy the equation:
\begin{gather}
    J(T_g\,x)=J(x),\quad J(x)\in C^{\infty}(V),
\end{gather}
are called \emph{representation invariants }. Introducing the operators:
\begin{gather*}
    t_A=-t_{Ab}^a\,x^b\,\partial_a,\quad t_{Ab}^a=\pder{(T_g)^a_b}{g^A}\Big|_{g=e},\quad A=\overline{1,\dim G},
\end{gather*}
this condition can be written in the form:
\begin{gather}\label{eq:002}
    t_A\, J(x)=0.
\end{gather}
Thus, the construction of a complete set of functionally independent invariants of the given representation reduces to the solution of the system (\ref{eq:002}), and the number $N$ of functionally independent invariants is defined by the equality:
\begin{gather}\label{eq:079}
    N=\corank\, t^{a}_{A}(x),
\end{gather}
where $t^{a}_{A}(x)=-t_{Ab}^a\,x^b$.

\p Consider the space $V^*$ dual to $V$, i.e.
\begin{gather*}
    V^*\times V\to \mathbb{C},\quad \la X, x\ra=\alpha,\quad x\in V,\ X\in V^*,\ \alpha\in \mathbb{C}.
\end{gather*}
Then the \emph{conjugate representation} $T^*(G)$, which we will also call \emph{copresentation}, is defined by the equality:
\begin{gather}\label{eq:003}
    \la T^*(g)\, X, T(g)\,x\ra=\la X, x\ra.
\end{gather}
Copresentation invariants are defined:
\begin{gather}\label{eq:004}
    t^*_A\, J^*(X)=0,\quad t^*_A=t^{a}_{Ab}\,X_a\,\partial^b.
\end{gather}
The number $N$ of functionally independent invariants of presentation $T^*(G)$ is defined by the equality:
\begin{gather*}
    N=\corank\, t_{Aa}(X),
\end{gather*}
where $t_{Aa}(X)=t_{Aa}^b\,X_b$.

\subsection{Relationship between invariants of representations and conjugate representations}
A possibility of elementary algebraic construction of Lie groups representations invariants is to find invariants of the conjugate representation with respect to the invariants of the original representation.

Let $G$ be a Lie group, $T(G)$ and $T^*(G)$ be some of its representations and the representation conjugate to it, acting in the spaces $V$ and $V^*$, respectively , and $J(x)$ and $J^*(X) $ are invariants of these representations which satisfy the equations (\ref{eq:002}) and (\ref{eq:004}). An arbitrary function $f(x)$ on the space of the representation $V$, satisfying the condition
\begin{gather*}
    \det\,\frac{\partial^2\, f(x)}{\partial\,x^a\,\partial\,x^b}\neq 0,\quad a,b=\overline{1,\dim V},
\end{gather*}
where $x\in V $ is the point of general position, we will call a \emph{non-degenerate function on the space $V$}. By analogy, we can define \emph{non-degenerate functions on the space $V^*$}.
\begin{Theorem}
Let $ J (x) $ be a non-degenerate invariant of the representation $T(G)$, then conjugate representation invariant can be found by formula:
\begin{gather}\label{eq:083}
    J^*(X)=\Le[x^a\,X_a- J(x)\Ri]\Big|_{X_a=\pder{J(x)}{x^a}}.
\end{gather}
\end{Theorem}
\begin{proof}
Due to the fact that $J(x) $ is the representation invariant $T(G)$, we have the equality
\begin{gather*}
    -t_{ia}^b\,x^a\,\pder{J(x)}{x^b}=0.
\end{gather*}
And due to the fact that $J^*(X)$ is the representation invariant $T^*(G)$ conjugate to $T(G)$, it must satisfy the condition (\ref{eq:004}):
\begin{gather*}
    t_{ia}^b\,X_b\,\pder{J^*(X)}{X_a}=0.
\end{gather*}
Then successively we get:
\begin{multline*}
    t_{ia}^b\,X_b\,\pder{J^*(X)}{X_a}=t_{ia}^b\,X_b\,\pder{}{X_a}\,\Le[x^c(X)\,X_c- J(x(X))\Ri]=\\=
    t_{ia}^b\,X_b\,\pder{x^c}{X_a}\,X_c+t_{ia}^b\,X_b\,x^c\pder{X_c}{X_a}-t_{ia}^b\,X_b\,\pder{J(x)}{x_c}\,\pder{x^c}{X^a}=\\=
    t_{ia}^b\,\pder{J(x)}{x^b}\,\pder{x^c}{X_a}\,\pder{J(x)}{x^c}+t_{ia}^b\,\pder{J(x)}{x^b}\,x^c\,\delta^a_c-t_{ia}^b\,\pder{J(x)}{x^b}\,\pder{J(x)}{x_c}\,\pder{x^c}{X^a}=\\=
    t_{ia}^b\,x^a\,\pder{J(x)}{x^b}=0.
\end{multline*}
\end{proof}
Note that formula (\ref{eq:083}) is known as the Legendre transformation. Also note that the inverse transform is defined by a similar formula:
\begin{gather*}
    J(x)=\Le[x^a\,X_a- J^*(X)\Ri]\Big|_{x^a=\pder{J^*(X)}{X_a}}.
\end{gather*}

\begin{Example}
Consider the Lie group $G$, the Lie algebra $\G=\{e_i\}$ which has commutation relations:
\begin{gather*}
    [e_2,e_3]=e_1,\quad
    [e_2,e_4]=e_3,\quad
    [e_3,e_4]=-e_2,
\end{gather*}
and some of its presentation:
\begin{gather*}
    g\to
    \left(
\begin{array}{cccc}
 1 & - s^3\,\cos  s^4-s^2\,\sin  s^4 & s^2\,\cos  s^4- s^3\sin  s^4 & \frac{(s^2)^2}{2}+\frac{(s^3)^2}{2} \\
 0 &  \cos  s^4 & \sin  s^4 & -s^3 \\
 0 & -\sin  s^4 & \cos  s^4 &  s^2 \\
 0 & 0 & 0 & 1 \\
\end{array}
\right)
\end{gather*}
This representation has two invariants, which are determined by the system:
\begin{gather*}
    x^4\,\pder{J(x)}{x^3}+x^3\,\pder{J(x)}{x^1}=0,\quad
    x^4\,\pder{J(x)}{x^2}+x^2\,\pder{J(x)}{x^1}=0,\quad
    x^3\,\pder{J(x)}{x^2}-x^2\,\pder{J(x)}{x^3}=0.
\end{gather*}
Its solutions have the form:
\begin{gather*}
    J_1(x)=(x^2)^2+(x^3)^2-2 x^1 x^4,\quad J_2(x)=x^4.
\end{gather*}
We construct a sheaf of invariants of the form:
\begin{gather*}
    J_\alpha(x)=\alpha_1\,J_1(x)+\alpha_2\,J_2(x),\quad \alpha=\{\alpha_1, \alpha_2\}.
\end{gather*}
It is easy to see by direct differentiation that the given function is invariant in the sense of the definition given above. Then there exists a one-to-one relationship between the coordinates $x^a$ on the space of the representation $V$ and the coordinates $X_a$ on the space of the conjugate representation:
\begin{gather*}
    X_1=\pder{J(x)}{x^1}=-2 \alpha_1  x^4,\quad
    X_2=\pder{J(x)}{x^2}= 2 \alpha_1  x^2,\\
    X_3=\pder{J(x)}{x^3}= 2 \alpha_1  x^3,\quad
    X_4=\pder{J(x)}{x^4}= \alpha_2 -2 \alpha_1  x^1.
\end{gather*}
These expressions can be easily inverted:
\begin{gather}\label{eq:022}
    x^1= \frac{\alpha_2 -X_4}{2 \alpha_1 },\quad
    x^2= \frac{X_2}{2 \alpha_1 },\quad
    x^3= \frac{X_3}{2 \alpha_1 },\quad
    x^4= -\frac{X_1}{2 \alpha_1 }.
\end{gather}
The sheaf of  representation invariants of the conjugate to a given is determined by the expression:
\begin{gather*}
    J^*_\alpha(X)=x^1 X_1+x^2 X_2+x^3 X_3+x^4 X_4-\alpha_1  \left((x^2)^2+(x^3)^2-2 x^1 x^4\right)-\alpha_2  x^4,
\end{gather*}
which taking into account (\ref{eq:022}) takes the form:
\begin{gather*}
    J^*_\alpha(X)=\frac{2 \alpha_2  X_1}{\alpha_1 }+\frac{X_2^2}{\alpha_1 }+\frac{X_3^2}{\alpha_1 }-\frac{2 X_1 X_4}{\alpha_1 },
\end{gather*}
and hence the required invariants of the conjugate representation have the form:
\begin{gather*}
    J^*_1(X)=X_2^2+X_3^2-2 X_1 X_4,\quad J^*_2(X)=X_1.
\end{gather*}
\end{Example}

\subsection{Construction of the Lie group representation invariants}
Traditionally, the construction of a Lie group representation invariants reduces to solving a system of differential equations (\ref{eq:002}). Nevertheless, it is possible to construct invariants without resorting to its direct integration.

The proposed method is a generalization of our earlier developed approach to the construction of invariants of the coadjoint representation of Lie groups by linear algebra methods, which is based on the fact that on a coalgebra that is a representation space for a coadjoint representation of a Lie group, there exists a Poisson structure. In the general case, there is no such structure on the representation space, but it is always possible to construct a special extension of a given space that admits a Poisson structure. This must be done in such a way that this extension can be regarded as a coalgebra of some suitable Lie algebra, and so for the construction of invariants one can use the method developed earlier for constructing invariants of the coadjoint representation.

Let $ G $ be a Lie group, and let $T(G)$ be its representation in the space $V$. Consider the space:
\begin{gather}\label{eq:084}
    \overline{\G}=\G \subsetplus V^*=\{e_A, E^a\},\quad \G=\{e_A\},\quad V^*=\{E^a\},\\
    \label{eq:085}
    [e_A, e_B]=C_{AB}^C\,e_C,\quad [e_A, E^a]=-t^{a}_{Ab}\,E^b,\quad [E^a, E^b]=0,
\end{gather}
where $A,B,C=\overline{1,\dim G},\ a,b=\overline{\dim M+1, \dim M+\dim G}$. It is not difficult to see that the structure constants of the space $\overline{\G} $ satisfy the Jacobi identity by construction, and hence the given space is a Lie algebra, which we shall henceforth refer to as the \emph{extended Lie algebra}.

Consider the Lie algebra $\overline{\G}^*$, which is a coalgebra of the Lie algebra $\overline{\G} $. On the coalgebra $\overline{\G}^*$ there is a natural Poisson structure of the following form:
\begin{gather*}
    \{X_A, X_B\}=C_{AB}^C\,X_C,\quad \{X_A, x^a\}=-t_{Ab}^a\,x^b=-\{x^a, X_A\},\quad \{x^a, x^b\}=0,
\end{gather*}
which corresponds to the functions of Casimir, determined by:
\begin{gather*}
    \{X_A,J(x,X)\}=C_{AB}^C\,X_C\,\pder{J(x,X)}{X_B}-t_{Ab}^a\,x^b\,\pder{J(x,X)}{x^a}=0,\\
    \{x^a,J(x,X)\}=-t_{Ab}^a\,x^b\,\pder{J(x,X)}{X_A}=0.
\end{gather*}
If $J(x, X)$ depends only on $x^a$, then this function is an invariant of the representation $T(G)$ on the space $V$. Thus we have the following theorem.
\begin{Theorem}
The problem of constructing invariants of an arbitrary representation of the Lie group $ G $ on the space $V$ is equivalent to the construction of invariants of the coadjoint representation of the extended Lie group $ \overline{G}$ whose Lie algebra $\overline{\G} $ is the semidirect sum of the Lie algebra $\G$ of the Lie group $G$ and the space $V^*$ dual to $V$, which can be regarded as a commutative ideal, i.e.
\begin{gather*}
    \overline{\G}=\G \subsetplus V^*.
\end{gather*}
\end{Theorem}

The reduction of the problem of constructing invariants of arbitrary representations to the calculation of invariants of a coadjoint representation is justified by the fact that the method of constructing invariants of the coadjoint representation, first published by the authors in the paper \cite{Shirokov:2016eng} and briefly stated in the previous section, makes it possible to find the last of these easily using linear algebra methods.

Thus, the practical recipe for constructing the required invariant is as follows: using the method outlined in \cite{Shirokov:2016eng}, we build a basis for Casimir functions, after that we construct a necessary number of functions that depend only on variables $x^a$ as a combination of these basis elements.

We note that, following the approach suggested above, the problem of constructing invariants of the conjugate representation does not need to be considered separately since conjugate representations are simply a particular case of representations.


\section{Example}
\noindent\textbf{Step 0: the problem statement}. We consider the Lie group $G$, $\dim G = 4$ and its Lie algebra $\G =\{e_A\}$, where $A=\overline{1,\dim G}$, with commutation relations:
\begin{gather}\label{eq:019}
    [e_1,e_4]=e_1+e_2,\quad
    [e_2,e_4]=e_2,\quad
    [e_3,e_4]=e_3,
\end{gather}
and some of its representation $T(G)$ in the space $V = \{E_a\} $, where $a=\overline {1, \dim V} $, of dimension $\dim V = 4 $ for which the corresponding representation $T(\G) $ of the Lie algebra $\G $ is defined:
\begin{gather}\label{eq:080}
    l_1\equiv T(e_1)= \left(
\begin{array}{cccc}
  0 &  0 & 0 & 0 \\
  0 &  0 & 0 & 0 \\
  0 &  0 & 0 & 0 \\
 -1 & -1 & 0 & 0 \\
\end{array}
\right),\quad l_2\equiv T(e_2)= \left(
\begin{array}{cccc}
 0 &  0 & 0 & 0 \\
 0 &  0 & 0 & 0 \\
 0 &  0 & 0 & 0 \\
 0 & -1 & 0 & 0 \\
\end{array}
\right),\\ l_3\equiv T(e_3)= \left(
\begin{array}{cccc}
 0 & 0 &  0 & 0 \\
 0 & 0 &  0 & 0 \\
 0 & 0 &  0 & 0 \\
 0 & 0 & -1 & 0 \\
\end{array}
\right),\quad l_4\equiv T(e_4)= \left(
\begin{array}{cccc}
 1 & 1 & 0 & 0 \\
 0 & 1 & 0 & 0 \\
 0 & 0 & 1 & 0 \\
 0 & 0 & 0 & 0 \\
\end{array}
\right)
\end{gather}
Invariants of this representation are defined by the equations:
\begin{gather*}
    (x^1+x^2)\,\pder{J}{x^4}=0,\quad x^2\,\pder{J}{x^4}=0,\quad x^3\,\pder{J}{x^4}=0,\\
     (x^1+x^2)\,\pder{J}{x^1}+x^2\,\pder{J}{x^2}+x^3\,\pder{J}{x^3}=0,
\end{gather*}
where $J(x)\in C^{\infty}(V)$. By (\ref{eq:079}), the number of functionally independent invariants $N = 2$.

\noindent\textbf{Step 1: constructing of the extended algebra}. Let us construct the extended algebra $\overline{\G}$, using formula (\ref{eq:084}) and (\ref{eq:085}), and replace for convenience the notation of the basis vectors of the representation space $E^a\to e_a$:
\begin{gather*}
    [e_1,e_4]=e_1+e_2,\quad
    [e_1,e_8]=e_5+e_6,\quad
    [e_2,e_4]=e_2,\quad
    [e_2,e_8]=e_6,\quad
    [e_3,e_4]=e_3,\\
    [e_3,e_8]=e_7,\quad
    [e_4,e_5]=-e_5-e_6,\quad
    [e_4,e_6]=-e_6,\quad
    [e_4,e_7]=-e_7.
\end{gather*}

\noindent\textbf{Step 2: constructing of the polarization}. We choose a non-degenerate covector of form $\lambda=(0, j_1, j_2, 0, j_3, j_4, 1, 0)$. Let us construct the polarization $ \mathcal {P} (\overline{\G}) $ of the extended Lie algebra $\overline{\G} $ with respect to a given covector. The annihilator of the chosen covector $\overline{\G}^{\lambda}$ is:
\begin{gather*}
    l_1=e_5-\frac{e_6 \left(j_3+j_4\right)}{j_4},\quad
    l_2= -\frac{e_2 \left(j_3+j_4\right)}{j_4}+\frac{e_6 j_1 j_3}{j_4^2}+e_1,\\
    l_3= e_7-\frac{e_6}{j_4},\quad
    l_4= -\frac{e_2}{j_4}-\frac{e_6 \left(j_2 j_4-j_1\right)}{j_4^2}+e_3.
\end{gather*}
Then the dimension of the polarization is $\dim \mathcal{P}(\overline{\G})=\frac{1}{2}(\dim \overline{\G}+\dim \overline{\G}^\lambda)=6$.

The Lie algebra $\overline{\G}$ is solvable, therefore, in order to construct the polarization, it is necessary to construct a chain of ideals of codimension 1:
\begin{gather*}
    \overline{\G}=\overline{\G}^{(8)}\supset \overline{\G}^{(7)}\supset\ldots \overline{\G}^{(2)}\supset \overline{\G}^{(1)}\supset \{0\}
\end{gather*}
The desired chain can be constructed by compaction of the upper central series:
\begin{gather}\label{eq:010}
    \left\{e_1,e_2,e_3,e_4,e_5,e_6,e_7,e_8\right\}\supset\left\{e_1,e_2,e_3,e_5,e_6,e_7\right\}\supset\left\{0\right\}.
\end{gather}
First of all, it is necessary to find a common eigenvector, which necessarily exists for every solvable Lie algebra. Having a common eigenvector, i.e. one-dimensional ideal, we can construct the corresponding quotient algebra. In it, it is also necessary to find a common eigenvector whose complete preimage will be an ideal greater than the previous by one, i.e. dimension two. Continuing by analogy, we obtain the desired chain of ideals:
\begin{gather*}
    \overline{\G}^{(1)}=\{e_7\},\quad \overline{\G}^{(2)}=\{e_6, e_7\},\quad \overline{\G}^{(3)}=\{e_5, e_6,e_7\},\quad \overline{\G}^{(4)}=\{e_3, e_5, e_6,e_7\},\\
    \overline{\G}^{(5)}=\{e_2, e_3, e_5, e_6,e_7\},\quad \overline{\G}^{(6)}=\{e_1, e_2, e_3, e_5, e_6,e_7\},\\
    \overline{\G}^{(7)}=\{e_4, e_1, e_2, e_3, e_5, e_6,e_7\},\quad \overline{\G}^{(8)}=\{e_8, e_4, e_1, e_2, e_3, e_5, e_6,e_7\}.
\end{gather*}
Since the subalgebras are commutative $\overline{\G}^{(1)},\ \overline{\G}^{(2)},\ \overline{\G}^{(3)},\ \overline{\G}^{(4)},\ \overline{\G}^{(5)},\ \overline{\G}^{(6)}$ , the corresponding covector annihilators $\lambda\big|_{\overline{\G}^{(i)}}$ are coincide with the very same subalgebras. For other elements of the series we have:
\begin{gather*}
    (\overline{\G}^{(7)})^\lambda=
    \left\{
    e_1-e_2,
    e_3-\frac{e_2 j_2}{j_1},
    e_6-\frac{e_2 j_4}{j_1},
    e_5-\frac{e_2 \left(j_3+j_4\right)}{j_1},
    e_7-\frac{e_2}{j_1}
    \right\},\\
    (\overline{\G}^{(8)})^\lambda=
    \left\{
    e_5-\frac{e_6 \left(j_3+j_4\right)}{j_4},\
    -\frac{e_2 \left(j_3+j_4\right)}{j_4}+\frac{e_6 j_1 j_3}{j_4^2}+e_1,\
    e_7-\frac{e_6}{j_4},\
    -\frac{e_2}{j_4}-\frac{e_6 \left(j_2 j_4-j_1\right)}{j_4^2}+e_3
     \right\}
\end{gather*}
Finally, for polarization of the algebra $\overline{\G} $, we obtain:
\begin{gather*}
    \mathcal{P}(\overline{\G})=(\overline{\G}^{(1)})^\lambda+(\overline{\G}^{(2)})^\lambda+(\overline{\G}^{(3)})^\lambda+(\overline{\G}^{(4)})^\lambda+
    (\overline{\G}^{(5)})^\lambda+(\overline{\G}^{(6)})^\lambda+(\overline{\G}^{(7)})^\lambda+(\overline{\G}^{(8)})^\lambda.\\
     \mathcal{P}(\overline{\G})=
     \left\{
     e_1-e_2,
     e_7,
     e_7-\frac{e_2}{j_1},
     e_3-\frac{e_2 j_2}{j_1},
     e_6-\frac{e_2 j_4}{j_1},
     e_5-\frac{e_2 \left(j_3+j_4\right)}{j_1}
     \right\}.
\end{gather*}

We choose a special basis in the algebra $\overline{\G}$ by adding the polarization $\mathcal{P}(\overline{\G})$ to the full space $ \overline{\G}$, i.e. the required basis must have the form $\{l_A, l_a\} $, where $\{l_A\} $ is the polarization basis of $\mathcal{P}(\overline{\G})$, and $\{l_a\}$ are additional basis vectors:
\begin{gather*}
    l_1=e_1-e_2,\
    l_2=e_7,\
    l_3=e_7-\frac{e_2}{j_1},\
    l_4=e_3-\frac{e_2 j_2}{j_1},\
    l_5=e_6-\frac{e_2 j_4}{j_1},\\
    l_6=e_5-\frac{e_2 \left(j_3+j_4\right)}{j_1},\
    l_7=e_4,\
    l_8=e_8.
\end{gather*}
with commutation relations:
\begin{gather*}
    \left[l_1,l_7\right]=j_1 l_2-j_1 l_3+l_1,\
    \left[l_1,l_8\right]=\left(j_3+j_4\right) l_2+\left(-j_3-j_4\right) l_3+l_6,\
    \left[l_2,l_7\right]=l_2,\
    \left[l_3,l_7\right]=l_3,\\
    \left[l_3,l_8\right]=-\frac{j_4 l_2}{j_1}+\frac{j_4 l_3}{j_1}-\frac{l_5}{j_1},\
    \left[l_4,l_7\right]=l_4,\
    \left[l_4,l_8\right]=\left(1-\frac{j_2 j_4}{j_1}\right) l_2+\frac{j_2 j_4 l_3}{j_1}-\frac{j_2 l_5}{j_1},\\
    \left[l_5,l_7\right]=l_5,\
    \left[l_5,l_8\right]=\frac{j_4^2 l_3}{j_1}-\frac{j_4^2 l_2}{j_1}-\frac{j_4 l_5}{j_1},\
    \left[l_6,l_7\right]=j_4 l_2-j_4 l_3+l_5+l_6,\\
    \left[l_6,l_8\right]=\left(-\frac{j_4^2}{j_1}-\frac{j_3 j_4}{j_1}\right) l_2+\left(\frac{j_4^2}{j_1}+\frac{j_3 j_4}{j_1}\right) l_3-\frac{j_3 l_5}{j_1}-\frac{j_4 l_5}{j_1}
\end{gather*}
The coordinates of the previously selected covector in this basis have the form:
\begin{gather}\label{eq:087}
    \Lambda _1= -j_1,\
    \Lambda _2= 1,\
    \Lambda _3= 0,\
    \Lambda _4= 0,\
    \Lambda _5= 0,\
    \Lambda _6= -j_4,\
    \Lambda _7= 0,\
    \Lambda _8= 0.
\end{gather}

\noindent\textbf{Step 3: constructing of the Darboux coordinates}. The transition to the Darboux coordinates is determined by:
\begin{gather*}
    F_1= e^{-q_7} \left(q_8 \left(j_3 \left(\Lambda _3-\Lambda _2\right)+j_4 \left(\Lambda _2-\Lambda _3\right) \left(q_7-1\right)-\Lambda _6+\Lambda _5 q_7\right)+j_1 \left(\Lambda _3-\Lambda _2\right) q_7+\Lambda _1\right),\\
    F_2= \Lambda _2 e^{-q_7},\\
    F_3= \frac{e^{-q_7} \left(j_1 \Lambda _3+j_4 \left(\Lambda _2-\Lambda _3\right) q_8+\Lambda _5 q_8\right)}{j_1},\\
    F_4= \frac{e^{-q_7} \left(j_1 \left(\Lambda _4-\Lambda _2 q_8\right)+j_2 q_8 \left(j_4 \left(\Lambda _2-\Lambda _3\right)+\Lambda _5\right)\right)}{j_1},\\
    F_5= \frac{e^{-q_7} \left(j_1 \Lambda _5+j_4^2 \left(\Lambda _2-\Lambda _3\right) q_8+j_4 \Lambda _5 q_8\right)}{j_1},\\
    F_6= \frac{e^{-q_7} \left(j_1 \left(j_4 \left(\Lambda _3-\Lambda _2\right) q_7+\Lambda _6-\Lambda _5 q_7\right)+\left(j_3+j_4\right) q_8 \left(j_4 \left(\Lambda _2-\Lambda _3\right)+\Lambda _5\right)\right)}{j_1},\\
    F_7= p_7,\\
    F_8= p_8
\end{gather*}
Let us return to the original basis in $\overline{\G}$ taking:
\begin{gather*}
    F_1= f_1-f_2,\
    F_2= f_7,\
    F_3= f_7-\frac{f_2}{j_1},\
    F_4= f_3-\frac{f_2 j_2}{j_1},\
    F_5= f_6-\frac{f_2 j_4}{j_1},\\
    F_6= f_5-\frac{f_2 \left(j_3+j_4\right)}{j_1},\
    F_7= f_4,\
    F_8= f_8
\end{gather*}
As a result, we get:
\begin{gather*}
    f_1= -e^{-q^7} \left(j_1 q^7-j_4 q^8 q^7+j_3 q^8+j_4 q^8\right),\
    f_2= e^{-q^7} \left(j_1-j_4 q^8\right),\\
    f_3= e^{-q^7} \left(j_2-q^8\right),\
    f_4= p_7,\\
    f_5= e^{-q^7} \left(j_3-j_4 q^7\right),\
    f_6= j_4 e^{-q^7},\
    f_7= e^{-q^7},\
    f_8= p_8.
\end{gather*}

\noindent\textbf{Step 4: eliminating the variables $q^a$ and $p_a$}. Eliminating the variables $q^a$ and $p_a$ in the obtained transition functions, we find the invariants of the coadjoint representation of the extended algebra $\overline{\G}$ have the form:
\begin{gather*}
    j_1= \frac{f_2 f_5-f_1 f_6+f_2 f_6}{f_5 f_7+f_6 f_7 (1-\ln f_7)},\\
    j_2=\frac{f_3}{f_7}-\frac{f_1- f_2 \ln f_7}{f_5+f_6 \left(1-\ln f_7\right)},\\
    j_3= \frac{f_5}{f_7}-\frac{f_6}{f_7}\,\ln f_7,\\
    j_4= \frac{f_6}{f_7}.
\end{gather*}

\noindent\textbf{Step 5: selection of the initial group representation invariants}. Invariants of the coadjoint representation of the extended Lie group $\overline{G} $ are arbitrary functions $K(j_1, j_2, j_3, j_4)$. The invariants of the representation of the original Lie group $G$ in the space $V$ are invariants that depend only on the coordinates $f_5,\ f_6,\ f_7,\ f_8$. In this case, such are $ K_1(j_1, j_2, j_3, j_4)=j_3$ and $K_2(j_1, j_2, j_3, j_4) = j_4$.

In the original notation (i.e., renaming the variables $f_5\to x^1$, $f_6\to x^2$, $f_7\to x^3$, $f_8\to x^4$) we finally get:
\begin{gather*}
    J_1(x)=\frac{x^2}{x^3},\quad
    J_2(x)=\frac{x^1}{x^3}-\frac{x^2}{x^3}\, \ln \left(x^3\right).
\end{gather*}


\bibliographystyle{unsrt}

\end{document}